\documentclass[11pt]{amsart}
\usepackage{amsthm,amsmath,amssymb,mathrsfs,yhmath,dsfont}
\usepackage{xcolor}
\usepackage{enumerate}
\newtheorem{thm}{Theorem}

\newtheorem{lem}[thm]{Lemma}

\theoremstyle{definition}

\theoremstyle{remark}

\newcommand{\Real}{\mathbb R}

\newcommand{\R}{\hbox{\ensuremath{\mathbb{R}}}}

\begin{document}
\title{On bounded distortions of maps in the line}
\author{Ignacio Garcia}
\address{Departamento de Matem\'atica (FCEyN-UNMDP), Mar del Plata, Argentina}
\email{nacholma@gmail.com}
\author{Carlos Gustavo Moreira}
\address{Instituto Nacional de Matem\'atica Pura e Aplicada, Estrada Dona Castorina 110, 22460-320, Rio de janeiro, Brasil}
\email{gugu@impa.br}
\begin{abstract}
We give an example illustrating that two notions of bounded
distortion for $\mathcal C^1$ expanding maps in $\R$ are different.
\end{abstract}
\maketitle
\section{Introduction and definitions}
Let $I_1$ and $I_2$ be disjoint closed intervals and let $F:I_1\cup
I_2\to [0,1]$ be a $\mathcal C^1$ map such that $F|_{I_i}$ is a
diffeomorphism on $[0,1]$ and $F'>1$ on its domain. The map $F$ has
associated a unique repeller $K$ given by
$K=\cap_{k\ge1}F^{-k}([0,1])$; the set $K$ is the maximal invariant
set under $F$ and is a Cantor set. Its Hausdorff and upper box
dimensions coincide, and they may be equal to $1$ (see \cite{PT},
Chapter $4$). More information on $K$ is obtained imposing
conditions on the map $F$. More precisely, let $\Omega_k$ be the set
of words of length $k$ with symbols $0$ and $1$, and note that the
$k$-th iterate $F^k$ is defined on the family of $2^k$ closed
intervals $\{I_\omega:\omega\in\Omega_k\}$, labeled from left to
right using the lexicographical order on $\Omega_k$. Note that the
restriction $F^k|_{I_\omega}$ is a diffeomorphism onto $[0,1]$. We
say that the map $F$ satisfies the {\em bounded distortion property}
{\bf BD} if there exists a constant $1\le C<\infty$ such that
\begin{equation*}
\frac{(F^k)'(x)}{(F^k)'(y)}\le C \ \ for \ all \ k>0,
\end{equation*}
and for all $x, y\in I_\omega$ and $\omega\in\Omega_k$. Moreover,
$F$ satisfies the {\em strong bounded distortion property} {\bf SBD}
if there is a sequence ${\beta_l}$ decreasing to $1$ such that
\begin{equation*}\label{SBD}
\frac{(F^k)'(x)}{(F^k)'(y)}\le \beta_r \ \ for \ all \ k>0,
\end{equation*}
whenever $x, y$ belong to the same basic interval $I_\omega, \omega
\in \Omega_k$ and $|F^k([x,y])|\le 1/r$, where $|A|$ denotes the diameter of the set $A$.

Clearly {\bf SBD} implies {\bf BD}. Moreover, it is well known that
if $F'$ is $\alpha$-H\"{o}lder continuous, then {\bf SBD} holds (see
for example \cite{PT}), and the same is true if the modulus of
continuity $w(t)=\sup_{|x-y|<t}|F'(x)-F'(y)|$ satisfies the Dini
condition  $\int_0^1w(t)t^{-1} dt$ (see \cite{FJ99}). Let $\dim_H K$
denotes the Hausdorff dimension of $K$. Property {\bf BD} implies
that $0<\dim_H K<1$, and also that the $\dim_H K$-dimensional
Hausdorff measure is positive and finite. Moreover, property {\bf
SBD} is needed, for example, to define the scaling function, which
is a $\mathcal C^1$ complete invariant for Cantor sets defined by
smooth maps: two such sets with the same scaling function are
diffeomorphic (see \cite{BeFi}).

However, although one suspects that {\bf SBD} is actually stronger
than {\bf BD}, we did not find in the literature an example
illustrating this fact. The purpose of this note is to provide such
an example.

\section{The example}

In order to construct $F$, we need a special family
$\{\varphi_t\}_{t\in[-1,1]}$ of smooth diffeomorphisms of the
interval $[0,1]$. For this reason, let $X$ be the $C^\infty$ field
on $[0,1]$ defined by $X(0)=X(1)=0$ and $X(x)=\exp((x(x-1))^{-1})$.
Consider its associated flow $\{\varphi_t\}_{t\in\Real}$ (see for
example \cite{Lang02}): for each $x\in[0,1]$ let $\tilde\phi(t,x)$
be the solution of the equation
$$\left\{\begin{matrix}
\frac{d}{dt}\phi(t,x)=X(\phi) \\
\hspace{-.4cm}\phi(0,x)=x
\end{matrix};\right.$$
then $\varphi_t=\tilde\phi(t,\cdot)$. Note $\tilde\phi\in\mathcal
C^\infty(\Real\times[0,1])$, which by the initial condition implies
$\varphi_t(0)=0$ and $\varphi_t(1)=1$ for all $t$. Below we list the
properties of $\{\varphi_t\}_{t\in[-1,1]}$ that we will use:
\begin{enumerate}[i)]
\item $\varphi_0(x)=x$ and $\varphi_t\circ\varphi_s=\varphi_{t+s}$, whenever $t, s, t+s\in[-1,1]$;
\item $\varphi_t'(0)=\varphi_t'(1)=1$, for all $t$;
\item $\|\varphi_t'-1\|_{u}\to0$ as $t\to0$;
\item $\phi_t'(x)\ge 2/3$, for all $x$ and $t\in[-T,T]$, for some $0<T\le1$;
\item there exists $M$ such that $\|\varphi_t''\|_u\le M$ for all $t \in [-1,1]$.
\end{enumerate}
Property i) is the semigroup property for flows; ii) follows from
the identity $(\partial/\partial t)(\partial/\partial
x)\tilde\phi(t,x)=X'(\tilde\phi(x,t))(\partial/\partial
x)\tilde\phi(t,x)$ and since $X'(0)=X'(1)=0$; iii) and v) are
consequence of the smoothness of $\tilde\phi$, while iv) follows
from iii).

For $n\ge0$, let $J_n=[2/3^{n+1},1/3^n]$ and denote by $A_n:J_n\to
[0,1]$ and $B_n:[0,1]\to J_{n-1}$ the affine maps
\begin{equation*}
A_n(x)=3^{n+1}x-2 \ \ and \ \ B_n(x)=\frac{x+2}{3^n}.
\end{equation*}
Note
\begin{equation}\label{eqid}
A_n\circ B_{n+1}=id_{[0,1]}.
\end{equation}
Also, for $2^k\le n<2^{k+1}$, let $t_n=(-1/2)^kT$. We define
$$F:[0,1/3]\cup[2/3,1]\to[0,1]$$ by
$$F(x)=\begin{cases}
                B_n\circ\varphi_{t_n}\circ A_n(x), & if \ x\in J_n, \ n\ge1 \\
                3x, & if \ x\in[0,1/3]\setminus\cup_{n\ge1}J_n \\
              3x-2, & if \ x\in [2/3,1]
        \end{cases}.    $$

Note that $F$ satisfies the diagram below.
\begin{figure*}[ht]
\begin{center}
\begin{picture}(140,80)
\put(5,0){$[0,1]$} \put(30,3){\vector(1,0){80}}
\put(60,-5){$\varphi_{t_n}$} \put(115,0){$[0,1]$}
\put(15,50){\vector(0,-1){40}} \put(-1,28){$A_n$} \put(10,55){$J_n$}
\put(125,11){\vector(0,1){40}} \put(128,28){$B_n$}
\put(115,55){$J_{n-1}$} \put(30,58){\vector(1,0){80}}
\put(62,63){$F$}
\end{picture}
\end{center}
\end{figure*}
\begin{lem}
The function $F$ is $\mathcal C^1$.
\end{lem}
\begin{proof}
Clearly $F'$ exists and is continuous on
$\bigl((0,1/3]\setminus\cup_{n\ge1}J_n\bigr)\cup[2/3,1]$. The
existence and continuity of $F'$ at $1/3^n$ and $2/3^n$ follows
computing the left and right-sided derivatives and using ii): both
are equal to $3$. Moreover, for the right hand sided derivative at
$0$, if $h\in J_n$, then by the mean value theorem,
$$
\left|\frac{F(h)-F(0)}{h}-3\right| 
=\left|3\varphi_{t_m}'(3^{m+1}\xi_h-2)-3\right|
=3\bigl|\varphi_{t_m}'(3^{m+1}\xi_h-2)-1\bigr|$$
(for some $m \ge n$ and $\xi_h\in J_m$), which tends to $0$ by iii).
This implies the existence of $F'(0)$ (right sided), and the
continuity at $0$ also follows from iii).
\end{proof}
Given $\omega=\omega_1\ldots\omega_n$ and $\tau=\tau_1\ldots\tau_k$
we define
$\omega\tau=\omega=\omega_1\ldots\omega_n\tau_1\ldots\tau_k$. Also,
let $0^n, 1^n\in\Omega_n$ be the words formed only by zeroes and
ones, respectively.

We need two preliminary lemmas.
\begin{lem}\label{lemma2}Let $\tau\in\Omega_k$.
\begin{enumerate}[a)]
\item $I_{0^n1}=J_n$ for all $n>0$. In particular, $I_{0^n1\tau}\subset J_n$. Moreover, if $x\in I_{0^n1\tau}$ and $t=t_1+\cdots+t_n$, then $t\in [0,T]$ and
\begin{equation*}
(F^{n})'(x)=3^{n}\varphi_t'(A_{n}(x)).
\end{equation*}
\item $I_{1^n\tau}\subset J_0$ for all $n>0$. Moreover, if $x\in I_{1^n\tau}$, then
\begin{equation*}
(F^n)'(x)=3^n.
\end{equation*}
\end{enumerate}
\end{lem}
\begin{proof}
We first notice that $\sum_{2^k \le
n<2^{k+1}}t_n=2^k\cdot(-1/2)^kT=(-1)^kT, \forall k\ge 0$, and so, if
$2^k \le n <2^{k+1}$, $k$ even, then
$t_1+\cdots+t_n=T-T+\dots+T-T+\sum_{2^k \le m \le
n}(-1/2)^kT=\sum_{2^k \le m \le n}(-1/2)^kT=\frac{n-2^k+1}{2^k}T\in
[0,T]$, and if $2^k \le n <2^{k+1}$, $k$ odd, then
$t_1+\cdots+t_n=T-T+\dots+T-T+T+\sum_{2^k \le m \le
n}(-1/2)^kT=T-\sum_{2^k \le m \le
n}(1/2)^kT=(1-\frac{n-2^k+1}{2^k})T\in [0,T]$.

By definition, $F$ is a bijection from $J_n$ to $J_{n-1}$.
Therefore, it can be shown inductively that $I_{0^n1}=J_n$, for all
$n>0$.
In particular, if $x\in J_{n}$ we have by i) and \eqref{eqid} that
\begin{equation}\label{eq2}
F^{n}(x)=B_{1}\circ\varphi_{t}\circ A_{n}(x),
\end{equation}
and differentiating we obtain part $a)$. Part $b)$ is immediate from
the definition of $F$.
\end{proof}
The following is an estimate on the size of basic intervals.
\begin{lem}
For  each  $n\ge0$  and  $\tau\in\Omega_k$ we have
\begin{equation*}
|I_{0^{n}1\tau}|\le3^{-n+1}2^{-k-2}.
\end{equation*}
\end{lem}
\begin{proof}
Denote by $f_{0^n1\tau}$ the inverse of $F^{k+n}|_{I_{0^n1\tau}}$,
which is a diffeomorphism onto $[0,1]$. Then, for some $\xi\in(0,1)$
we have
$$|I_{0^{n}1\tau}|=f'_{0^{n}1\tau}(\xi)=\frac{1}{(F^{k+n+1})'(f_{0^{n}1\tau}(\xi))}.$$
Set $y:=f_{0^{n}1\tau}(\xi)\in I_{0^{n}1\tau}$. Then, by Lemma
\ref{lemma2} $a)$ and iv) we have
\begin{align*}
(F^{k+n+1})'(y)&=(F^{k+1})'(F^{n}(y))(F^n)'(y) \\
   &=3^n\varphi_{t_1+\cdots+t_n}'(A_n(y))(F^{k+1})'(F^{n}(y)) \\
   &\ge\frac{2}{3}3^n2^{k+1},
\end{align*}
and the lemma follows.
\end{proof}

Now we are ready to verify that $F$ satisfies {\bf BD} but not {\bf
SBD}.\vspace{.2cm}

\noindent {\em $F$ does not satisfies {\bf SBD}.} Let $\alpha, \beta
\in [0,1]$ be such that $\varphi_T'(\alpha)\neq\varphi_T'(\beta)$
(they exist since $\varphi_T \neq Id$). Observe that for each $k$
we have
$$F^{2^k}|_{J_{{2^{k+1}-1}}}=B_{2^k}\circ\varphi_{(-1)^kT}\circ
A_{2^{k+1}-1}.$$ Then, if $k$ is even and if $x, y\in
J_{{2^{k+1}-1}}$ are such that $A_{2^{k+1}-1}(x)=\alpha$ and
$A_{2^{k+1}-1}(y)=\beta$, we obtain
\begin{equation*}
\frac{(F^{2^k})'(x)}{(F^{2^k})'(y)}=
\frac{\varphi'_{T}(A_{2^{k+1}-1}(x))}{\varphi'_{T}(A_{2^{k+1}-1}(y))}=
\frac{\varphi'_{T}(\alpha)}{\varphi'_{T}(\beta)}\neq 1,
\end{equation*}
whence {\bf SBD} does not hold; indeed,
$F^{2^k}(J_{2^{k+1}-1})=J_{2^k-1}$, whose size tends to $0$ when $k
\to \infty$.\vspace{.2cm}

\noindent {\em $F$ satisfies {\bf BD}.} Fix $k>0$ and
$\omega\in\Omega_k$ and let $x, y\in I_\omega$. We consider the
blocks of zeroes and ones of $\omega$, that is, there is an $L>0$
such that $\omega=0^{m_1}1^{n_1}0^{m_2}\cdots 0^{m_L}1^{n_L}$,
where $m_j, n_j>0$ for all $j$ but possibly $m_1=0$ or $n_L=0$ (the
case in which $\omega$ begins with $1$ or ends with $0$,
respectively).
We have $F^k(x)=F^{n_L}\circ F^{m_L}\circ\cdots\circ F^{n_1}\circ
F^{m_1}(x)$. Then for each $j$, $F^{m_j}$ is evaluated at a point
$x_j\in I_{0^{m_j}\tau_j}\subset J_{m_j}$, where
$|\tau_j|=n_j+\sum_{i=j+1}^L(m_i+n_i)$, hence by Lemma \ref{lemma2}
$a)$,
$$(F^{m_j})'(x_j)=3^{m_j}\varphi_{\ell_j}'(A_{m_j}(x_j))$$
for some $\ell_j \in [-T,T]$. Moreover, $F^{n_j}$ is evaluated at a
point $\tilde x_j\in I_{1^{n_j}\gamma}$, where
$|\gamma|=\sum_{i=j+1}^L(m_i+n_i)$, hence by Lemma \ref{lemma2}
$b)$,  $(F^{n_j})'(\tilde x_j)=3^{n_j}$.
Therefore
\begin{align*}
\frac{(F^k)'(x)}{(F^k)'(y)}&=\prod_{j=1}^L\frac{(F^{m_j})'(x_j)}{(F^{m_j})'(y_j)}  \\
      &=\prod_{j=1}^L\frac{\varphi_{\ell_j}'(A_{m_j}(x_j))}{\varphi_{\ell_j}'(A_{m_j}(y_j))} \\
      &=\prod_{j=1}^L \left(1+\frac{\varphi_{\ell_j}'(A_{m_j}(x_j))-\varphi_{\ell_j}'(A_{m_j}(y_j))}{\varphi_{\ell_j}'(A_{m_j}(y_j))}\right) \\
     &=\prod_{j=1}^L  \left(1+\frac{\varphi_{\ell_j}''(\xi_j)}{\varphi_{\ell_j}'(A_{m_j}(y_j))}3^{m_j+1}(x_j-y_j)\right) \\
     &\le\prod_{j=1}^L\left(1+\frac{3^3M}{2^{|\tau_j|+2}}\right),
\end{align*}
the inequality follows from iv), v) and Lemma 3, since $|x_j-y_j|\le
|I_{0^{m_j}\tau_j}|$.
The last product is uniformly bounded since
$\sum_{j=1}^L2^{-|\tau_j|}\le\sum_{i=0}^\infty2^{-i}<\infty$.
Therefore $F$ satisfies {\bf BD}.

\section*{Acknowledgments}I G is partially supported by CAI+D2009 $N^\circ$
62-310. C G M

\end{document}